\RequirePackage{fix-cm}
\documentclass[notitlepage,11pt]{amsart}
\usepackage{microtype}
\usepackage{fixltx2e}[2006/09/13]
\usepackage{mathtools}
\usepackage[firstinits]{biblatex}

\usepackage{enumitem}
\usepackage{hyperref}

\usepackage{tikz}
\usetikzlibrary{decorations.pathreplacing}

\newtheorem{thm}{Theorem}
\newtheorem{lmm}{Lemma}
\newtheorem{crl}{Corollary}
\newtheorem*{crlp}{Corollary 2$'$}

\theoremstyle{definition}
\newtheorem*{dfn}{Definition}
\newtheorem*{qst}{Question}

\DeclareMathOperator{\aff}{aff}
\DeclareMathOperator{\relint}{relint}

\newcommand{\N}{\mathbb{N}}
\newcommand{\Z}{\mathbb{Z}}
\newcommand{\Q}{\mathbb{Q}}
\newcommand{\R}{\mathbb{R}}
\newcommand{\C}{\mathbb{C}}
\newcommand{\E}{\mathbb{E}}
\newcommand{\sB}{\mathcal{B}}
\newcommand{\sC}{\mathcal{C}}

\newcommand{\sL}{\mathcal{L}}

\newcommand{\sT}{\mathcal{T}}

\newcommand{\enleve}{\setminus}

\DeclarePairedDelimiter{\abs}{\lvert}{\rvert}  % absolute value, cardinality

\DeclarePairedDelimiter{\grpgen}{\langle}{\rangle}
\newcommand{\midbar}{\mathrel{}\mathclose{}\delimsize|\mathopen{}\mathrel{}} % For use in paired delimiters
\DeclarePairedDelimiterX{\scprd}[2]{\langle}{\rangle}{#1,#2}
\DeclarePairedDelimiterX{\set}[2]{\{}{\}}{\,#1 \midbar #2\,}
\DeclarePairedDelimiterX{\grppres}[2]{\langle}{\rangle}{#1 \midbar #2}

\newcommand*{\stak}[2]{\genfrac{}{}{0pt}{}{#1}{#2}}

\newcommand{\res}{\kern -0.1em\mathpunct{\upharpoonright}}

\title{Combinatorially two-orbit convex polytopes}
\author[Matteo]{Nicholas Matteo}
\address{Department of Mathematics, 567 Lake Hall\\Northeastern University\\Boston, MA 02115}
\email{matteo.n@husky.neu.edu}
\date{\today}

\begin{document}
\begin{abstract}
Any convex polytope whose combinatorial automorphism group has two orbits on the 
flags is isomorphic to one whose group of Euclidean symmetries has two orbits on 
the flags (equivalently, to one whose automorphism group and symmetry group 
coincide). Hence, a combinatorially two-orbit convex polytope is isomorphic to 
one of a known finite list, all of which are 3-dimensional: the cuboctahedron, 
icosidodecahedron, rhombic dodecahedron, or rhombic triacontahedron. The same is 
true of combinatorially two-orbit normal face-to-face tilings by convex 
polytopes.
\end{abstract}
\maketitle

\section{Introduction}
Regular polytopes are those whose symmetry groups act transitively on their 
flags (see Section~\ref{prelim} for definitions; throughout this article, 
``polytope'' means convex polytope.)
We say that a polytope whose symmetry group has $k$ orbits on the flags is a 
\emph{$k$-orbit polytope}, so the regular polytopes are the one-orbit polytopes. The
one-orbit polytopes in the plane (the regular polygons) and in 3-space (the
Platonic solids) have been known for millenia; the six one-orbit 4-polytopes and
the three one-orbit $d$-polytopes for every $d \geq 5$ have been known since the 
19th century.
In \cite{matteo2014two}, the author found all the two-orbit polytopes.
These exist only in the plane and in 3-space.
In the plane, there are two infinite families, one consisting of the irregular
isogonal polygons, and the other consisting of the irregular isotoxal polygons.
Here, \emph{isogonal} means that the symmetry group acts transitively on the 
vertices, and \emph{isotoxal} means that the symmetry group acts transitively on 
the edges.
In 3-space, there are only four: the two quasiregular polyhedra,
namely the cuboctahedron and the icosidodecahedron, and their duals,
the rhombic dodecahedron and the rhombic triacontahedron.

A polytope is \emph{combinatorially $m$-orbit} if its automorphism group has $m$ orbits
on the flags. In general, a polytope has more combinatorial automorphisms
of its face lattice than it has Euclidean symmetries.
Hence, if the symmetry group has $k$ flag orbits
and the automorphism goup has $m$ flag orbits,
then $m \leq k$; in fact $m \mid k$.
Furthermore, not every polytope can be realized
such that every automorphism is also a Euclidean isometry;
\cite{bokowski1984combinatorial} constructs a combinatorially 84-orbit
4-polytope $P$
which is not isomorphic to any polytope $P'$ whose symmetry group $G(P')$
is equal to the automorphism group $\Gamma(P')$.
However, it is proved in \cite[Theorem 3A1]{McMThesis}
that a polytope is combinatorially one-orbit if and only if
it is isomorphic to a (geometrically) one-orbit polytope.
In this paper, we show that every combinatorially two-orbit polytope
is isomorphic to a (geometrically) two-orbit polytope. The converse is not quite true, 
since any $2n$-gon is isomorphic to a two-orbit polytope,
yet is not combinatorially two-orbit.

In Section~\ref{tilings} we show, similarly, that combinatorially two-orbit
normal face-to-face tilings by convex polytopes are isomorphic to two-orbit tilings.
It seems that the corresponding question for one-orbit tilings remains open,
with a finite list of possible exceptions.
We summarize the results in these theorems.

\begin{thm}\label{thm:combintopes}
Any combinatorially two-orbit convex polytope
is isomorphic to a (geometrically) two-orbit convex polytope.
Hence,
if $P$ is a combinatorially two-orbit convex $d$-polytope,
then $d = 3$ and $P$ is isomorphic to one of
the cuboctahedron, the icosidodecahedron,
the rhombic dodecahedron, or the rhombic triacontahedron.
\end{thm}

In light of the fact that, for $d > 2$, all two-orbit convex $d$-polytopes
are combinatorially two-orbit,
and that both conditions are vacuous for $d \leq 1$,
we can say that a convex $d$-polytope with $d \neq 2$ is combinatorially two-orbit
if and only if it is isomorphic to a two-orbit convex polytope.
%This does not hold in the plane,
%since every $2n$-gon ($n \geq 2$) is isomorphic to a two-orbit convex polygon,
%yet every polygon is combinatorially regular.

\begin{thm}\label{thm:lftilings}
A locally finite, combinatorially two-orbit tiling by convex polytopes
need not be isomorphic to a two-orbit tiling by convex polytopes.
However, locally finite, combinatorially two-orbit tilings of $\E^d$ by convex polytopes
only occur for $d = 2$ or $d = 3$.
\end{thm}

Terms related to tilings (such as ``normal'') are defined at the
beginning of Section~\ref{tilings}.

\begin{thm}\label{thm:normtilings}
Any combinatorially two-orbit, normal tiling by convex polytopes
is isomorphic to a two-orbit tiling by convex polytopes.
Hence, if $\sT$ is a combinatorially two-orbit normal tiling of $\E^d$ by convex polytopes, then either
\begin{enumerate}[label=\upshape(\roman*)]
\item $d=2$ and $\sT$ is isomorphic to one of
the trihexagonal tiling or the rhombille tiling, or
\item $d=3$ and $\sT$ is isomorphic to one of
the tetrahedral-octahedral honeycomb or the rhombic dodecahedral honeycomb.
\end{enumerate}
\end{thm}

\section{Preliminaries}\label{prelim}
We briefly review the terminology used. See \cite{ARP,grunbaum1967convex,coxeter1973regular}
for details.
\subsection{Basic terminology for polytopes}
A \emph{convex polytope} is the convex hull of finitely many points in $\E^d$.
Throughout this article, ``polytope,'' unqualified, means ``convex polytope.''
The dimension of a polytope is the dimension of its affine hull;
a polytope $P$ of dimension $d$ is called a $d$-polytope,
and the faces of $P$ with dimension $i$ are its \emph{$i$-faces}.
The 0-faces are called \emph{vertices}, 1-faces are called \emph{edges},
$(d-2)$-faces are called \emph{ridges}, and $(d-1)$-faces are called \emph{facets}.
In addition to these \emph{proper} faces, we admit two \emph{improper} faces,
namely a $(-1)$-face (the \emph{empty face}) and a $d$-face (which is $P$ itself).
With the inclusion of these improper faces, the faces of $P$ ordered by inclusion
form a lattice, the \emph{face lattice of $P$}, denoted $\sL(P)$.
A \emph{flag} of $P$ is a maximal chain (linearly ordered subset) in $\sL(P)$.
For any flag $\Phi$, an \emph{adjacent} flag is one which differs from $\Phi$ in exactly one face.
The flags are \emph{$i$-adjacent} if they differ in only the $i$-face.
Every flag $\Phi$ has a unique $i$-adjacent flag for $i = 0, \dotsc, d-1$,
denoted $\Phi^i$ (this is due to the ``diamond condition'' on polytopes.)
Two faces are said to be \emph{incident} if one contains the other.
A \emph{section} of $P$, for incident faces $F \subset G$,
is the portion of the face lattice $\sL(P)$ consisting of all the faces containing $F$ and contained in $G$,
and is denoted $G / F$. So $G / F = \set{H \in \sL(P)}{F \leq H \leq G}$, inheriting the order.
Every such section can be realized as the face lattice of a convex polytope,
and we will often identify convex polytopes with their face lattices.

For a convex polytope $P$, 
the \emph{symmetries} of $P$
are the Euclidean isometries which carry $P$ onto itself,
and form a group denoted $G(P)$.
The \emph{automorphisms} of $P$ are inclusion-preserving bijections
from the face lattice of $P$ to itself, and form a group denoted $\Gamma(P)$.
Each symmetry of $P$ also acts on the faces of $P$ in an inclusion-preserving manner,
so we can identify $G(P)$ with a subgroup of $\Gamma(P)$.
A $d$-polytope is said to be \emph{fully transitive} if its symmetry group acts
transitively on its $i$-faces for every $i = 0, \dotsc, d-1$.
It is \emph{combinatorially fully transitive} if its automorphism group
acts transitively on the faces of each dimension. In this case we may instead say that $\Gamma(P)$
is fully transitive.

\subsection{Class}
Let $I \subset \{0,1,\dotsc,d-1\}$ and
$\Phi$ be a flag of a combinatorially two-orbit $d$-polytope $P$.
If the $i$-adjacent flag $\Phi^i$
is in the same orbit as $\Phi$ if and only if $i \in I$,
then we say $P$ is in class $2_I$ \cite{hubard2010two,hubard2009monodromy}.
It is not hard to see that this class is well-defined; see \cite{hubard2010two} for proofs of this
and the following remarks.
The automorphism group $\Gamma(P)$ is fully transitive if and only if $\abs{I} \leq d - 2$.
We cannot have $\abs{I} = d$, because then $P$ would be combinatorially regular.
The only other case is that $\abs{I} = d - 1$, and then $\Gamma(P)$ acts transitively
on all $i$-faces with $i \in I$, but has two orbits on the $j$-faces for the unique $j$ not in $I$.
\begin{dfn}
An (abstract) polytope $P$ is \emph{$j$-intransitive}
if $\Gamma(P)$ does not act transitively on the $j$-faces,
but acts transitively on the $i$-faces for all $i \neq j$.
\end{dfn}
We shall see that all two-orbit polytopes are $j$-intransitive for some $j$.

\subsection{Modified Schl\"afli symbol}
The Schl\"afli symbol of a polytope is a standard concept; see e.g.\@ \cite[11]{ARP},\cite{mcmullen1967combinatorially}, or \cite{coxeter1973regular}.
For a regular $d$-polytope $P$, it is a list of $d-1$ numbers, 
$\{p_1, \dotsc, p_{d-1}\}$, where $p_i$ is the order of 
the automorphism $(\rho_{i-1} \rho_i)$,
where $\rho_k$ is an involution which carries a base flag $\Phi$ to its $k$-adjacent flag $\Phi^k$.
For convex polytopes, it is equivalent to say that for every incident pair
of an $(i-2)$-face $F_{i-2}$ and an $(i+1)$-face $F_{i+1}$,
the section $F_{i+1}/F_{i-2}$ is a $p_i$-gon. This is the meaning we will focus on.

For the purposes of the article, we will use a modified Schl\"afli symbol.
It is like the standard symbol $\{p_1, \dotsc, p_{d-1}\}$ for a $d$-polytope $P$,
but possibly with some positions $p_i$ replaced by a stack of two distinct numbers, $\stak{p_i}{q_i}$.
Wherever a single number $p_j$ appears,
it means (as usual) that every section $F_{j+1}/F_{j-2}$ is a $p_j$-gon.
If two numbers $\stak{p_j}{q_j}$ appear, it means that all such sections
$F_{j+1}/F_{j-2}$ are either $p_j$-gons or $q_j$-gons.
If $P$ is a two-orbit polytope with such a symbol,
then the orbit of a flag $\Phi = \{F_{-1}, \dotsc, F_d\}$
is determined by whether $F_{j+1}/F_{j-2}$ is a $p_j$-gon
or a $q_j$-gon.
If it is a $p_j$-gon, and $P$ is of class $2_I$,
then the corresponding section of $\Phi^i$ is a $q_j$-gon precisely
when $i \notin I$.
In order for the section to have a different size,
$\Phi^i$ must differ from $\Phi$ in either the $(j+1)$-face
or the $(j-2)$-face---but by definition it differs in exactly the $i$-face.
We conclude that $(j + 1)$ or $(j-2)$ (or both) are not in $I$.

Beware that you cannot read off the symbols for sections from the symbol for $P$,
as you can with a standard Schl\"afli symbol, without additional information.
For instance, in the type $\{4,\stak 3 4, 4\}$ discussed below,
the facets are of type $\{4,\stak 3 4\}$ (the rhombic dodecahedron)
and the vertex figures are of type $\{\stak 3 4, 4\}$ (the cuboctahedron).
However, in the tetrahedral-octahedral tiling of type $\{3, \stak 3 4, 4\}$,
the vertex figures are cuboctahedra $\{\stak 3 4, 4\}$,
but the facets alternate between two types, tetrahedra $\{3, 3\}$ and octahedra $\{3, 4\}$.

Those polytopes with standard Schl\"afli symbols (with just one number in each position),
so that the size of every section $F_{j+1}/F_{j-2}$ depends only on $j$,
are called \emph{equivelar}.
Equivelar convex polytopes are combinatorially regular \cite[Theorem 1B9]{ARP}.
On the other hand, in a combinatorially two-orbit polytope,
obviously the sections $F_{j+1}/F_{j-2}$ for a given $j$ can have at most two sizes.
So every combinatorially two-orbit convex polytope has a modified Schl\"afli symbol,
with at least one stack of two numbers appearing.

\subsection{Results on combinatorially two-orbit polytopes}
\label{results}
For a $d$-polytope $P$ and $I \subseteq \{-1,0,\dotsc,d\}$,
a \emph{chain of type $I$} is a chain of faces in $\sL(P)$
with an $i$-face for each $i \in I$, and no others.
A \emph{chain of cotype $I$} is a chain in $\sL(P)$
with an $i$-face for each $i \notin I$, and no others.

\begin{lmm}\label{cochain-trans}
If $P$ is in class $2_I$ and $j \notin I$,
then $\Gamma(P)$ acts transitively on chains of cotype $\{j\}$.
\end{lmm}
\begin{proof}
Let $\Psi$ and $\Omega$ be two chains of cotype $\{j\}$.
Each of these may be extended to two flags of $P$
which, being $j$-adjacent, are in different flag orbits.
Thus, we extend $\Psi$ to a flag $\Psi'$
and $\Omega$ to a flag $\Omega'$ such that both are in the same orbit;
then the automorphism $\gamma \in \Gamma(P)$ carrying $\Psi'$ to $\Omega'$
also takes $\Psi$ to $\Omega$.
\end{proof}

\begin{crl}\label{singleval}
If $P$ is in class $2_I$ and $j \notin I$,
then $P$ has a modified Schl\"afli symbol
whose entry $p_i$ is single-valued
except possibly at $i=j-1$ and $i=j+2$.
\end{crl}
\begin{proof}
By Lemma~\ref{cochain-trans}, $\Gamma(P)$ acts transitively on the
sections $F_{i+1}/F_{i-2}$ for each rank $i$ unless $i + 1 = j$ or $i - 2 = j$. 
\end{proof}

Recall that if all entries of the Schl\"afli symbol are single-valued,
then $P$ is combinatorially regular.
But by the Corollary,
if two distinct ranks $i < j$ are missing from $I$,
then all the entries would be single-valued
unless $j = i + 3$, so that $j - 1$ coincides with $i + 2$.
This also shows that no three ranks $i < j < k$ can be missing from $I$.

\begin{lmm}\label{evenval}
If a $d$-polytope $P$ is in class $2_I$ and $j \notin I$,
then the entries $p_j$ (if $j \geq 1$) and $p_{j+1}$ (if $j \leq d -2$)
are even.
\end{lmm}
\begin{proof}
If $1 \leq j \leq d - 1$,
then consider any section $F_{j+1}/F_{j-2}$ with incident faces
of the indicated ranks.
This is a polygon whose edges correspond to $j$-faces of $P$.
A walk along the edges of this polygon
can be extended to a sequence of adjacent flags of $P$,
alternately $j$-adjacent and $(j-1)$-adjacent.
The flags change orbits whenever the $j$-face is changed.
But changing $(j-1)$-faces (corresponding to vertices of the polygon)
will not change the orbit, since $(j-1)$ and $j$ do not differ by 3.
Thus the polygon has evenly many sides.
Hence $p_j$, the $j$th entry in the Schl\"afli symbol for $P$
(which is single-valued by Corollary~\ref{singleval}) is even.

Similarly,
if $0 \leq j \leq d - 2$, then
any section $F_{j+2}/F_{j-1}$ is a polygon whose vertices correspond
to $j$-faces of $P$.
A walk along the edges of this polygon
corresponds to a sequence of adjacent flags of $P$,
alternately $j$-adjacent or $(j+1)$-adjacent,
with the $j$-adjacent flags in different orbits.
Hence the polygon again has evenly many sides, 
so $p_{j+1}$ is even.
\end{proof}

\begin{crl}\label{vertorfacet}
If a $d$-polytope $P$ is in class $2_I$ and $j \notin I$,
then $j = 0$ or $j = d - 1$.
\end{crl}
\begin{proof}
If $j \notin I$ and $0 < j < d-1$,
then both the entries $p_j$ and $p_{j+1}$
appear in the Schl\"afli symbol.
But this contradicts Euler's theorem;
a polyhedral section $F_{j+2}/F_{j-2}$
would have the symbol $\{p_j,p_{j+1}\}$
with two even entries, which is impossible for convex polytopes
\cite[\S 13.1]{grunbaum1967convex}.
\end{proof}

Continuing the preceding remarks,
we conclude that the only way two distinct ranks can be
missing from $I$, where $P$ is in class $2_I$,
is if $I$ omits both 0 and $d-1$ and $d - 1 = 0 + 3$,
i.e.\@ $P$ must be a 4-polytope in class $2_{\{1,2\}}$.
We will postpone considering this special case until Section~\ref{4polys}.
Otherwise, $\abs{I} = d-1$ and any two-orbit polytope
of type $2_I$ must be either vertex-intransitive
or facet-intransitive. Since vertex-intransitive polytopes
are the duals of the facet-intransitive polytopes,
we will deal with the latter in Section~\ref{combtwo}.

\section{Combinatorially facet-intransitive two-orbit polytopes}\label{combtwo}
Suppose $P$ is a combinatorially two-orbit $d$-polytope
which is facet-intransitive, i.e.\@ it is in class $2_I$ where
$I = \{0, 1, \dotsc, d-2\}$. 
It follows that $P$ is what is called an \emph{alternating semiregular polytope}
in \cite{monson2012semiregular}.
Fix a flag $\Phi$, the \emph{base flag}.
Then for each $i \in I$, there is an automorphism $\rho_i \in \Gamma(P)$
such that $\rho_i(\Phi) = \Phi^i$.
There is no automorphism carrying $\Phi$ to $\Phi^{d-1}$, which is in the other orbit.
However, the flag $\Phi^{d-1,d-2,d-1}$, reached by changing the facet of $\Phi$,
then changing the ridge, then flipping facets again, is in the same orbit as $\Phi$,
so there is an automorphism $\rho'_{d-2}$ carrying $\Phi$ to $\Phi^{d-1,d-2,d-1}$.
This automorphism is referred to as $\alpha_{d-1,d-2,d-1}$ in \cite{hubard2010two}.

\begin{lmm}\label{generators}
The automorphisms $\rho_i$ and $\rho'_{d-2}$ generate the whole automorphism group of $P$,
so $\Gamma(P) = \grpgen{\rho_0, \rho_1, \dotsc, \rho_{d-2}, \rho'_{d-2}}$.
\end{lmm}
\begin{proof}
Write $\Phi = \{F_{-1}, F_0, \dotsc, F_{d-1}, P\}$,
and say the facet-adjacent flag $\Phi^{d-1}$
has the facet $F_{d-1}'$.
First we show that the given generators suffice
to carry the flag $\Phi^{d-1}$ to each of its adjacent flags
$\Phi^{d-1,i}$ for $i \leq d- 2$.

Let $i \leq d - 3$.
Since $\rho_i$ fixes $F_{d-2}$ and $F_{d-1}$, it must also fix $F_{d-1}'$.
Hence, it fixes all faces of $\Phi^{d-1}$ except for its $i$-face $F_i$;
so $\rho_i(\Phi^{d-1}) = \Phi^{d-1,i}$.

On the other hand, $\rho_{d-2}$ cannot fix $F_{d-1}'$.
Since $\rho_{d-2}(\Phi) = \Phi^{d-2}$,
the image of the $(d-1)$-adjacent flag $\Phi^{d-1}$
must be $(d-1)$-adjacent to $\Phi^{d-2}$,
i.e.\@ $\rho_{d-2}(\Phi^{d-1}) = \Phi^{d-2,d-1}$.
But the automorphism $\rho'_{d-2}$
which carries $\Phi$ to $\Phi^{d-1,d-2,d-1}$
must carry $\Phi^{d-1}$ to $\Phi^{d-1,d-2}$.

Thus, the given generators carry $\Phi^{d-1}$ to each of
its adjacent flags except for $\Phi$.

Now let $\gamma$ be any automorphism of $P$.
The automorphism $\gamma$ is the unique one carrying
$\Phi$ to $\gamma(\Phi)$.
By exhibiting an automorphism in $\grpgen{\rho_0, \rho_1, \dotsc, \rho_{d-2}, \rho'_{d-2}}$ carrying $\Phi$ to $\gamma(\Phi)$,
we show that the arbitrary element $\gamma$ lies in this subgroup.

By the flag-connectedness property of polytopes,
there is a sequence of adjacent flags
$\Phi = \Phi_0, \Phi_1, \dotsc, \Phi_n = \gamma(\Phi)$.
For each $k \in \{1,\dotsc,n\}$ there is some $i_k \in \{0,\dotsc, d-1\}$
such that the flag $\Phi_k$ is $i_k$-adjacent to the preceding flag $\Phi_{k-1}$.
Suppose $1 \leq k \leq n$ and we have written either $\Phi_{k-1} = \sigma(\Phi)$
or $\Phi_{k-1} = \sigma(\Phi^{d-1})$ for some $\sigma \in \grpgen{\rho_0, \rho_1, \dotsc, \rho_{d-2}, \rho'_{d-2}}$.

If $i_k \leq d - 3$, then $\Phi_k$ is
$\sigma(\rho_{i_k}(\Phi))$ or $\sigma(\rho_{i_k}(\Phi^{d-1}))$, respectively.

If $i_k = d - 2$, then $\Phi_k$ is
$\sigma(\rho_{d-2}(\Phi))$ or $\sigma(\rho'_{d-2}(\Phi^{d-1}))$, respectively.

If $i_k = d - 1$, then $\Phi_k$ is
$\sigma(\Phi^{d-1})$ or $\sigma(\Phi)$, respectively.

Thus we continue until we have written
$\Phi_n = \sigma(\Phi)$ or $\sigma(\Phi^{d-1})$ for some
$\sigma \in \grpgen{\rho_0, \rho_1, \dotsc, \rho_{d-2}, \rho'_{d-2}}$.
Since $\Phi_n = \gamma(\Phi)$ is in the same orbit as $\Phi$
and $\Phi^{d-1}$ is not,
we must in fact have $\Phi_n = \sigma(\Phi)$ and $\gamma = \sigma$.
\end{proof}

By Corollary~\ref{singleval} with $j = d - 1$,
$P$ will have a modified Schl\"afli symbol of the form $\{p_1,\dotsc,p_{d-3},\stak{p_{d-2}}{q_{d-2}},p_{d-1}\}$,
where $p_{d-2} \neq q_{d-2}$, since $P$ cannot be equivelar.
Figure~\ref{coxdiagram} shows the Coxeter diagram for these generators,
modified by labeling the nodes with the corresponding generator.
Such a diagram is dubbed a ``tail-triangle diagram'' in \cite{monson2012semiregular},
making $\Gamma(P)$ a ``tail-triangle group.''
Note that $p_{d-1}$ must be even, by Lemma~\ref{evenval}.

\begin{figure}[h]
\begin{tikzpicture}[xscale=1.2]
\draw (0,0) node[fill,circle,inner sep=1.5pt,label=below:$\scriptstyle\rho_0$] {} -- node[above] {$p_1$}
	  (1.3,0) node[fill,circle,inner sep=1.5pt,label=below:$\scriptstyle\rho_1$] (r1) {}
	  (3,0)  node[fill,circle,inner sep=1.5pt,label=below:$\scriptstyle\rho_{d-4}$] (rn) {} -- node[above] {$p_{d-3}$}
	  (4.3,0) node[fill,circle,inner sep=1.5pt,label=below:$\scriptstyle\rho_{d-3}$] (fr) {} -- node[above=2pt] {$p_{d-2}$}
	  (5.6,.9) node[fill,circle,inner sep=1.5pt,label=right:$\scriptstyle\rho_{d-2}$] (tp) {};
\draw (fr) -- node[below] {$q_{d-2}$}
      (5.6,-.9) node[fill,circle,inner sep=1.5pt,label=right:$\scriptstyle\rho'_{d-2}$] (bt) {};
\draw (bt) -- node[right] {$\frac{p_{d-1}}{2}$} (tp);
\draw[very thick,loosely dotted] (r1) -- (rn);
\end{tikzpicture}
\caption{The Coxeter diagram of a facet-intransitive two-orbit polytope}\label{coxdiagram}
\end{figure}

Since the generators of $\Gamma(P)$ satisfy all the Coxeter relations implied by the diagram,
$\Gamma(P)$ is a quotient of the Coxeter group associated with the diagram.
However, in principle the generators of $\Gamma(P)$ might satisfy additional relations.
We shall show that, in fact, there are no additional relations in $\Gamma(P)$,
so that $\Gamma(P)$ is exactly the Coxeter group associated with the diagram
in Figure~\ref{coxdiagram}. Since $\Gamma(P)$ is finite,
we can then have recourse to the classification of finite Coxeter groups.

\begin{lmm}\label{Coxgroup}
The automorphism group $\Gamma(P)$ is a Coxeter group, with
Coxeter diagram as in Figure~\ref{coxdiagram}.
\end{lmm}
The proof is a modification of that of \cite[Theorem 3A1]{McMThesis},
that a combinatorially regular convex polytope is isomorphic to a regular one. 
The method is also in Coxeter's proof \cite[\S 5.3]{coxeter1973regular} that the Coxeter relations fully define the group generated
by reflections in the walls of the fundamental region described by the diagram.
The essence is that any relation in the group
(i.e.\@ a word in the generators representing the identity) can be represented as
a loop in the boundary of the polytope $P$;
contracting this loop to a point gives a guide to reducing the word,
using the given relations, until it is empty.
This shows that every relation in the group is a consequence
of the Coxeter relations.
The following proof is modeled on, and sometimes verbatim from, \cite[\S 5.3]{coxeter1973regular}.

\begin{proof}
We associate flags of $P$ with chambers of a 
``barycentric subdivision'' $\sB$ of the boundary of $P$.
Each flag $\Omega = \{G_{-1}, G_0, \dotsc, G_{d-1}, G_d\}$
is associated to the $(d-1)$-simplex whose vertices
are ``barycenters'' of each proper face of $\Omega$.
These barycenters can be any preassigned points in the relative interior of each face of $P$.
So the vertices of the simplex for $\Omega$
are the vertex $G_0$, the midpoint (say) of the edge $G_1$, and so on
up an interior point of the facet $G_{d-1}$.
Each face of this simplex corresponds to a subchain of $\Omega$.
A facet of the simplex is a $(d-2)$-simplex involving
the centers of all but one of the proper faces in $\Omega$.
Say the missing face is $G_i$. Then the facet,
called the \emph{$i$th wall}, forms the boundary
between the simplex for $\Omega$ and the simplex for the $i$-adjacent flag
$\Omega^i$. We identify each flag with its corresponding chamber
in the boundary of $P$.

The union of the chambers $\Phi$ and $\Phi^{d-1}$ constitute
a ``fundamental region'' $R$ for $\Gamma(P)$,
since every flag is the image of one of these.
For $0 \leq i \leq d - 3$,
the $i$th wall of $\Phi$ and the $i$th wall of $\Phi^{d-1}$ are contiguous,
and we will call their union the $i$th wall of $R$.
The $(d-2)$th wall of $\Phi$ is called the $(d-2)$th wall of $R$,
and the $(d-2)$th wall of $\Phi^{d-1}$ is called the $z$th wall of $R$
($z$ is just a symbol distinct from $0,\dotsc,d-1$.)
The $(d-1)$th walls of $\Phi$ and $\Phi^{d-1}$ are in the interior of $R$
and are not walls of $R$.
Thus, $R$ has walls labeled $0, \dotsc, d-2, z$.
%See Figure~\ref{fig:glued} for low-dimensional examples.

Say the vertex of $\sB$ lying in the relative interior of a face $F_i$
of $\Phi$ is $C_i$.
Recall that $F'_{d-1}$ is the facet in $\Phi^{d-1}$; say
the vertex in $\relint(F'_{d-1})$ is $C'_{d-1}$.
Then $R$ contains the $d$ vertices $C_i$, plus $C'_{d-1}$,
but $C_{d-2}$ is on the ``edge'' from $C'_{d-1}$ to $C_{d-1}$
and is not a vertex of $R$, so that $R$ has $d$ vertices
and is again a simplex, with vertices $C_0, \dotsc, C_{d-3}, C_{d-1}, C'_{d-1}$.
See Figure~\ref{fig:glued}.
(Some facets may be skew, rather than linear.)
In the left figure, where $d = 3$, the $i$th wall is labeled $i$.
In the right figure, where $d = 4$,
the face $C_0 C_1 C'_3$ is the $z$th wall;
the face $C_0 C_1 C_3$ is the 2nd wall;
the face $C_0 C'_3 C_3$ is the 1st wall;
and the face $C_1 C'_3 C_3$ is the 0th wall.

\begin{figure}[h]
\begin{tikzpicture}[scale=1.5]
\draw (0,0) node[left] {$C_0$} -- node[below left] {1} 
     (1,-1) node[right] {$C_2$} -- 
     (1,0) node[right] {$C_1$} -- 
     (1,1) node[right] {$C'_2$} -- node[above left] {$z$} (0,0);
\draw[dashed] (0,0) -- (1,0);
\draw[decorate,decoration={brace,amplitude=6pt}] (1.4,.9) -- node[right=3pt]{0} (1.4,-.9);
\end{tikzpicture}
\hfill
\begin{tikzpicture}[scale=2]
\draw (0,1,-1) node[above] {$C_0$} -- 
     (1.5,0,0) node[right] {$C_3$} -- 
     (0,0,-1) node[above right] {$C_1$} -- 
     (-1.5,0,0) node[left] {$C'_3$} -- (0,1,-1) -- (0,0,-1);
\draw (-1.5,0,0) -- (0,0,0) node[below] {$C_2$} -- (1.5,0,0);
\draw[dashed] (0,1,-1) -- (0,0,0) -- (0,0,-1);
\end{tikzpicture}
\caption{The region $R$ composed of $(d-1)$-adjacent chambers, in the case $d = 3$ (left) or $d = 4$ (right)}\label{fig:glued}
\end{figure}

Now for $\gamma \in \Gamma(P)$,
the chambers for $\gamma(\Phi)$ and $\gamma(\Phi^{d-1})$ are adjacent,
and their union is called ``region $\gamma$.''
We pass through the $i$th wall of region $\gamma$
into region $\gamma \rho_i$ (for $i \in \{0,\dotsc,d-2,z\}$),
where $\rho_z$ denotes $\rho'_{d-2}$.
Each automorphism $\gamma$ carries $i$-faces to other $i$-faces,
and the two orbits of $(d-1)$-faces are carried only to themselves.
Although $\gamma$ does not actually map points to other points,
if we consider a vertex of $\sB$ as representing the face in whose relative interior
it lies, it makes sense to say that each vertex of $R$ is carried only to
the unique vertex of the same type in each region $\gamma$.

To a word $w = \rho_{i_1}\dots\rho_{i_k}$,
where $i_j \in \{0,\dotsc,d-2,z\}$, 
we associate a path from $R$ to region $\rho_{i_1}\dots\rho_{i_k}$
passing through the $i_1$th wall of $R$,
then the $i_2$th wall of region $\rho_{i_1}$, and so on.
(By a path we mean a continuous curve which avoids any $(d-3)$-face of $\sB$.)

If the word $w$ represents the identity,
we must show that the relation $w = 1$ is a consequence of the Coxeter relations
inherent in Figure~\ref{coxdiagram}.
These relations are $(\rho_i \rho_j)^{m_{ij}} = 1$,
where $m_{ii} = 1$ for all $i$, and otherwise $m_{ij}$ is the label
on the edge from $\rho_i$ to $\rho_j$, or 2 if there is no edge.
If $w = 1$, the path associated to $w$ is a closed path back to $R$.
Consider what happens to the expression $\rho_{i_1}\dots\rho_{i_k}$
as the closed path is gradually shrunk until it lies wholly within region $R$.
Whenever the path goes from one region to another and then immediately returns,
this detour may be eliminated by canceling a repeated $\rho_i$ in the expression,
in accordance with the relation $(\rho_i)^2 = 1$.
The only other kind of change that can occur during the shrinking process
is when the path momentarily crosses a $(d-3)$-face $F$.

If $F$ is the intersection of the $i$th and $j$th walls of one region,
so that it does not contains vertices of the types opposite the $i$th and $j$th walls,
then $F$ does not contain vertices of those types in any region that contains it.
So the walls containing $F$ alternate between $i$th walls and $j$th walls,
and $F$ is contained in $2m_{ij}$ regions.

This change will replace $\rho_i \rho_j \rho_i \dotsm$ by $\rho_j \rho_i \rho_j \dotsm$ (or vice versa) in accordance with the relation $(\rho_i \rho_j)^{m_{ij}} = 1$.
The shrinkage of the path thus corresponds to an algebraic reduction of the expression $w$ by means of the Coxeter relations.
Since the boundary of $P$ is topologically a $(d-1)$-sphere,
and simply connected if $d > 2$, we can shrink the path to a point.
It follows that every relation in $\Gamma(P)$ is a Coxeter relation.
\end{proof}

We can now complete the proof of 
\begin{thm}\label{Comb2isGeo2}
Any combinatorially two-orbit facet-intransitive convex polytope
is isomorphic to a two-orbit convex polytope.
\end{thm}
Since $P$ has finitely many flags,
we know that $\Gamma(P)$ is a finite Coxeter group.
Consulting the list of finite Coxeter groups,
we see that $p_{d-1}/2$ must be 2, since no loops appear
in diagrams of finite Coxeter groups.
Furthermore, the only diagram with four or more nodes
that branches as in Figure~\ref{coxdiagram}
is $D_n$, where every edge has the label 3.
But we must have $p_{d-2} \neq q_{d-2}$, since $P$
is not equivelar.
Hence the diagram must not have a ``tail'':
we must have $d = 3$, and the only admissible diagrams
of finite Coxeter groups are those in Figure~\ref{gooddiagrams}.

\begin{figure}[h]
\begin{tikzpicture}[xscale=1.2]
\draw (0,0) node[fill,circle,inner sep=1.5pt,label=below:$\scriptstyle\rho_0$] (fr) {} -- node[above=2pt] {}
	  (1,1) node[fill,circle,inner sep=1.5pt,label=right:$\scriptstyle\rho_1$] (tp) {};
\draw (fr) -- node[below] {$4$}
      (1,-1) node[fill,circle,inner sep=1.5pt,label=right:$\scriptstyle\rho'_1$] (bt) {};
\node at (.5,-1.5) {$B_3 = C_3$};
\end{tikzpicture}
\hspace{1cm}
\begin{tikzpicture}[xscale=1.2]
\draw (0,0) node[fill,circle,inner sep=1.5pt,label=below:$\scriptstyle\rho_0$] (fr) {} -- node[above=2pt] {}
	  (1,1) node[fill,circle,inner sep=1.5pt,label=right:$\scriptstyle\rho_1$] (tp) {};
\draw (fr) -- node[below] {$5$}
      (1,-1) node[fill,circle,inner sep=1.5pt,label=right:$\scriptstyle\rho'_1$] (bt) {};
\node at (.5,-1.5) {$H_3$};
\end{tikzpicture}
\caption{Potential Coxeter diagrams for the automorphism group of a two-orbit polytope}\label{gooddiagrams}
\end{figure}

We know that both of these Coxeter groups occur as the automorphism group
of a two-orbit facet-intransitive polyhedron: $B_3$ for the cuboctahedron,
and $H_3$ for the icosidodecahedron.
The next lemma will show that the isomorphism type of a
2-orbit facet-intransitive polytope
is determined by its automorphism group (as a Coxeter system), so
these are the only possibilities.

For the purposes of the Lemma, we will fix a canonical
form of the Coxeter group presentation, as encoded in the
diagram of Figure~\ref{coxdiagram} or the Schl\"afli symbol
$\{p_1,\dotsc,p_{d-3},\stak{p_{d-2}}{q_{d-2}},p_{d-1}\}$,
such that $p_{d-2} < q_{d-2}$.
A flag $\Phi$ will be said to be an \emph{appropriate base flag}
if the generators $\rho_i$ corresponding to $\Phi$,
defined as in Lemma~\ref{generators},
satisfy $(\rho_{d-3}\rho_{d-2})^{p_{d-2}} = 1$.
We prove the Lemma generally,
rather than restricting to the two presentations in Figure~\ref{gooddiagrams},
so that it also applies to tilings, or indeed, any abstract polytopes.

\begin{lmm}\label{SameGpSameTope}
Two combinatorially two-orbit facet-intransitive polytopes
$P_1$ and $P_2$ are isomorphic if and only if their automorphism
groups have the same presentation
(with generators as in Lemma~\ref{generators},
and Coxeter relations as depicted in Figure~\ref{coxdiagram}),
if we require $p_{d-2} < q_{d-2}$.
\end{lmm}
\begin{proof}
If $h \colon \sL(P_1) \to \sL(P_2)$ is an isomorphism,
let $\Phi$ be an appropriate base flag for $P_1$.
Then the generators of $\Gamma(P_2)$ corresponding
to the base flag $h(\Phi)$ must satisfy the same relations
that the generators of $\Gamma(P_1)$ corresponding
to $\Phi$ do, so that the groups have the same presentation.

Conversely, suppose $P_1$ and $P_2$ are
combinatorially two-orbit facet-intransitive polytopes
with the same presentation.
For $i = 1,2$,
let $\Phi_i$ be an appropriate base flag of $P_i$
and define the generators
$\rho^i_0, \dotsc, \rho^i_{d-2}, \rho^{i \prime}_{d-2}$ of $\Gamma(P_i)$
with respect to $\Phi_i$ as in Lemma~\ref{generators}.
Since $\Gamma(P_1)$ and $\Gamma(P_2)$ have the same presentation,
the map carrying $\rho^1_j \mapsto \rho^2_j$ and
$\rho^{1\prime}_{d-2} \mapsto \rho^{2\prime}_{d-2}$ extends to a group isomorphism $f$.
Then the bijection 
of the sets of flags
taking $\gamma(\Phi_1) \mapsto f(\gamma)(\Phi_2)$
and $\gamma(\Phi_1^{d-1}) \mapsto f(\gamma)(\Phi_2^{d-1})$,
for all $\gamma \in \Gamma(P_1)$,
determines the required isomorphism between the lattices $\sL(P_1)$ and $\sL(P_2)$.
\end{proof}

\section{Exceptional possibilities in \texorpdfstring{$\E^4$}{E\textasciicircum 4}}
\label{4polys}
We now return to the exceptional possibilities left open for combinatorially 
two-orbit 4-polytopes
(see the end of Section~\ref{prelim}.) Recall that such a polytope $P$ is in class $2_{\{1,2\}}$, so 
it is combinatorially fully transitive. For any flag $\Phi$, the 1-adjacent flag 
$\Phi^1$ and 2-adjacent flag $\Phi^2$ are in the same orbit as $\Phi$, while the 
0-adjacent flag $\Phi^0$ and 3-adjacent flag $\Phi^3$ are not. By
2-face-transitivity, all the 2-faces have the same number of sides, $p_1$. All the 
edges are in the same number of facets, $p_3$. By Lemma~\ref{evenval}
with $j = 0$ and $j = 3$, $p_1$ and 
$p_3$ are even. Since $P$ is not equivelar, the Schl\"afli symbol has the form $
\{p_1, \stak{p_2}{q_2}, p_3\}$ where $p_1$ and $p_3$ are even.

Each facet and vertex figure of $P$ has at most two combinatorial flag orbits, by the action of the automorphism group of $P$ restricted to these sections.
Since $P$ is facet-transitive, each facet must have both $p_2$-gons and $q_2$-gons as vertex figures. Since $P$ is vertex-transitive, each
vertex figure must have both $p_2$-gons and $q_2$-gons as faces.
Thus the facets and vertex figures are not combinatorially regular:
they are combinatorially two-orbit 3-polytopes.
By the preceding proof, the facets and vertex figures are isomorphic to one of the four two-orbit polyhedra. Since all 2-faces are the same, and by the necessary compatibility of the vertex figures with the facets, the two possibilities are:
\begin{itemize}
\item
A polytope whose facets are isomorphic to the rhombic dodecahedron,
and whose vertex figures are isomorphic to cuboctahedra;
the modified Schl\"afli symbol is $\{4,\stak{3}{4},4\}$, and
\item
A polytope whose facets are isomorphic to the rhombic triacontahedron,
and whose vertex figures are isomorphic to icosidodecahedra;
the modified Schl\"afli symbol is $\{4,\stak{3}{5},4\}$.
\end{itemize}
In each case, the polytope would be combinatorially self-dual. However, we demonstrate that such polytopes cannot exist.

Suppose that $P$ has the first combinatorial type above, $\{4,\stak{3}{4},4\}$.
Consider the angle at a vertex $v$ in a 2-face $F$ containing $v$.
That is, in the affine hull $\aff(F)$, which is a plane, we take the interior angle of the quadrilateral $F$ at $v$.
The sum of all these angles at the 4 vertices of $F$ is $2\pi$.
So, if we take the sum of all such angles in the whole polytope $P$---i.e.\@ the sum of the angle
for every incident pair of vertex and 2-face in $P$---the sum is $2\pi f_2$, where $f_2$ is the number of 2-faces of $P$. Since every vertex is in 24 2-faces (the number of edges of the cuboctahedron),
and each 2-face has 4 vertices, $f_2 = 6 f_0$ (where $f_0$ is the number of vertices of $P$).

On the other hand, let $v$ be any vertex of $P$ and consider the sum of the
angles in each 2-face incident to $v$.
Each 2-face lies in exactly two facets: one where $v$ is in 4 edges, and one where $v$ is in 3 edges.
(Correspondingly, each edge of the vertex figure, the cuboctahedron, is in one square and one triangle.)
We may partition the 24 2-faces at $v$ into 6 sets of 4, each set consisting of the 2-faces of a facet $G$
containing $v$ wherein $v$ has valence 4.
The sum of the angles of $v$ within these four 2-faces must be less than $2\pi$ (the difference from $2\pi$ is the angular deficiency or defect.) Hence the sum of the angles at $v$ in all the 2-faces containing $v$
is less than $6 \cdot 2\pi$,
and the sum of the angles of all incident pairs of vertices and 2-faces is therefore less than $6 f_0 2 \pi$.

But this contradicts the earlier conclusion that the sum is exactly $6 f_0 2 \pi$.
Therefore, no such polytope can exist.

The same argument rules out the possibility of a polytope of the second type, $\{4,\stak{3}{5},4\}$.
Each vertex is in 60 2-faces (the number of edges of the icosidodecahedron),
and each 2-face has 4 vertices, so we have $f_2 = 15 f_0$,
and the sum of the angles over all incident pairs of vertex and 2-face
is $15 f_0 2\pi$.

On the other hand, the 2-faces at each vertex $v$ can be partitioned
into 12 sets of 5, each set consisting of the 2-faces of a particular facet $G$
containing $v$ wherein $v$ has valence 5.
The sum of the angles at $v$ in all these 2-faces is less than $2\pi$,
so the sum of all the angles of $v$ in the 60 2-faces containing $v$
is less than $12 \cdot 2\pi$.

Thus we have $15 f_0 2 \pi < 12 f_0 2\pi$, a contradiction, so no such polytope can exist.

With these possibilities disposed of, we have proved Theorem~\ref{thm:combintopes}.

\section{Tilings}\label{tilings}
In this section, we deal with combinatorially two-orbit tilings of Euclidean space $\E^d$. All the tilings we consider are by convex polytopes, and are \emph{face-to-face}, which means that the intersection of any two tiles is a face of each (possibly the empty face).
A tiling is \emph{locally finite} if every bounded set meets only finitely many tiles.

\begin{dfn}
An \emph{LFC tiling} is a face-to-face
locally finite tiling of $\E^{d}$ by convex $d$-polytopes.
\end{dfn}
An LFC tiling of $\E^d$ is an abstract polytope of rank $d+1$.
A tiling is \emph{normal} if it satisfies three conditions:
\begin{enumerate}[label=N.\arabic*]
\item\label{topodisk} Every tile is a topological ball.
\item\label{intersect} The intersection of every two tiles is connected (or empty).
\item\label{bounded} The tiles are uniformly bounded.
That is, there are positive numbers $u$ and $U$
such that every tile contains a ball of radius $u$ and is contained in a ball of radius $U$.
\end{enumerate}
Any convex tiling automatically satisfies properties
\ref{topodisk} and \ref{intersect}. So when we require a tiling to be normal,
it is equivalent to require the tile sizes to be bounded.
Every normal tiling is locally finite.

In Section~\ref{results},
Lemmas \ref{cochain-trans} and \ref{evenval}
and Corollary~\ref{singleval} apply also to combinatorially 2-orbit LFC tilings.
Corollary~\ref{vertorfacet} holds, but requires a modified proof:
\begin{crlp}
If a rank-$d$ LFC tiling $\sT$ is in class $2_I$ and $j \notin I$,
then $j = 0$ or $j = d - 1$.
\end{crlp}
\begin{proof}
Suppose $0 < j < d-1$.
Then there is an incident pair of faces $F_{j+2}$ and $F_{j-2}$.
If $F_{j+2}/F_{j-2}$ is a proper section,
then it is isomorphic to a convex polytope,
with symbol $\{p_j,p_{j+1}\}$
with two even entries (by Lemma~\ref{evenval}).
It is impossible for convex polytopes to have such a symbol
\cite[\S 13.1]{grunbaum1967convex}.

On the other hand, if $F_{j+2}/F_{j-2}$ is all of $\sT$,
then $j - 2 = -1$ and $j+2=d$, i.e.\@ $j = 1$ and $d = 3$,
so we have an edge-intransitive planar tiling.
By Corollary~\ref{singleval}, $\sT$ has type $\{p_1,p_2\}$
with single-valued entries. Hence $\sT$ is a regular tiling, a contradiction.
\end{proof}

Lemmas \ref{generators}, \ref{Coxgroup}, and \ref{SameGpSameTope} in Section~\ref{combtwo}
also apply to LFC tilings, but since the automorphism group of a tiling is not finite,
we get no corresponding short list of potential diagrams.
If we could conclude that the automorphism group was of so-called
``affine type,''
%then the only possible diagrams would be $\wt{G}_2$ and $\wt{B}_3$.
then the analog to Theorem~\ref{Comb2isGeo2} would follow.

Theorem 4A4 of McMullen's thesis \cite{McMThesis}
says, for $d \neq 3$, a rank-$d$ convex polytope with combinatorially regular vertex figures
and combinatorially regular facets is combinatorially regular.
The proof works equally well for LFC tilings;
we sketch it here.

\begin{lmm}\label{CombRegVerfsFacets}
A rank-$d$ LFC tiling $\sT$, $d \neq 3$, whose vertex figures
and facets are all combinatorially regular
is itself combinatorially regular.
\end{lmm}
\begin{proof}
Each vertex figure $\sT/v$ is combinatorially regular, hence equivelar,
so for any $i \in \{2,\dotsc,d-1\}$,
every incident pair of $(i+1)$-face $G_{i+1}$ and $(i-2)$-face $G_{i-2}$ containing $v$
gives a polygonal section $G_{i+1}/G_{i-2}$ of the same size, say $p_i(v)$.
Each edge figure, being contained in a vertex figure, is also equivelar,
so for any $i \in \{3,\dotsc,d-1\}$,
the size of each section $G_{i+1}/G_{i-2}$ of incident faces containing
a given edge is also constant.
Since any two vertices of $\sT$ may be linked by a chain of vertices
and edges,
the Schl\"afli entries $p_i$ with $i \geq 3$ are well-defined on all of $\sT$.

Similarly, face-chains of facets and ridges show that
$p_i$ is well-defined for $i \leq d-3$,
and face-chains of vertices and facets cover the remaining case
when $d = 4$ and $i = 2$.
\end{proof}

\begin{thm}
All combinatorially 2-orbit LFC tilings are of $\E^2$ or $\E^3$.
\end{thm}
\begin{proof}
A combinatorially two-orbit tiling has facets and vertex figures which are either combinatorially regular or combinatorially two-orbit.  If we are tiling $\E^d$, and $d \geq 4$, then by Theorem~\ref{thm:combintopes}
the facets and vertex figures are actually combinatorially regular, so the tiling is combinatorially regular. Thus $d < 4$.

Of course, LFC tilings of $\E^0$ and $\E^1$ are trivial, and no combinatorially two-orbit ones exist.
The remaining cases are tilings of $\E^2$ or $\E^3$.
\end{proof}

\subsection{Planar tilings}\label{plane}
Planar tilings are the only case, in light of Lemma~\ref{CombRegVerfsFacets},
where a combinatorially two-orbit tiling
can have combinatorially regular tiles and vertex figures. Indeed, any planar tiling has combinatorially regular
tiles and vertex figures, since all polygons are combinatorially regular.

\begin{thm}
There are infinitely many (isomorphism classes of)
combinatorially two-orbit LFC tilings of the plane.
\end{thm}
To see this, first we show that there are combinatorially regular tilings by convex $p$-gons,
with three tiles at each vertex, for every $p \geq 6$.
This is a consequence of result 4.7.1 of \emph{Tilings and Patterns}
\cite[194]{grunbaum1986tilings}.
We paraphrase the result, taking advantage of result 4.1.1 \cite{grunbaum1986tilings}
that homeomorphisms preserving a tiling are equivalent to combinatorial automorphisms,
and of convexification \cite[202]{grunbaum1986tilings}.

\begin{lmm}[{\cite[\nopp 4.7.1]{grunbaum1986tilings}}]
There exists a combinatorially regular LFC tiling of type $\{j,k\}$,
for positive integers $j, k$, if and only if $1/j + 1/k \leq 1/2$. Such a
tiling can be normal only if equality holds.
\end{lmm}

Since $1/3 + 1/p \leq 1/2$ for every $p \geq 6$,
there is a combinatorially regular tiling $\{p,3\}$
for every such $p$.
From this tiling, we can form a combinatorially two-orbit
tiling by ``truncating'' at each vertex to the midpoints of its incident edges,
analogously to the formation of the cuboctahedron from the cube,
of the icosidodecahedron from the dodecahedron,
or of the trihexagonal tiling from the regular hexagonal tiling.
Each edge of $\{p,3\}$ is reduced to its midpoint.
The midpoints of the three edges incident to a vertex
become the vertices of a triangular tile.
The midpoints of the $p$ edges of a $p$-gonal tile
in $\{p,3\}$ become the vertices of a smaller $p$-gonal tile.
For instance, with $p = 7$, this is a ``triheptagonal'' tiling.
Each vertex of this new tiling (formerly an edge midpoint)
is in four tiles: two triangles (the vertex figures of the endpoints
of the former edge), and two $p$-gons.
Thus the tiling can be described $(3.p.3.p)$,
a notation that gives, in cyclic order, the number of sides of each tile
incident to a vertex of the tiling.

However, none of these examples are normal
for $p \geq 7$.
We proceed to show
\begin{thm}
Every normal combinatorially two-orbit planar tiling
is isomorphic to one of the (geometrically) two-orbit planar tilings:
either the trihexagonal tiling or its dual, the rhombille tiling.
\end{thm}

By Corollary~$2'$,
a combinatorially two-orbit tiling $\sT$ of $\E^2$ is either facet-intransitive,
in which case $\Gamma(\sT)$ acts transitively on its vertices,
or vertex-intransitive, in which case $\Gamma(\sT)$ acts transitively on its facets (tiles).

In the former case,
we apply 
\begin{lmm}[{\cite[\nopp 3.5.4]{grunbaum1986tilings}}]
If every vertex of a normal tiling $\sT$ has valence $j$,
and is incident with tiles which have $k_1, \dotsc, k_j$ adjacents,
then
\[
\sum_{i=1}^j \frac{k_i - 2}{k_i} = 2.
\]
\end{lmm}
Each vertex is incident to evenly many tiles
(by Lemma~\ref{evenval} with $I=\{0,1\}$),
which alternate orbits.
With 6 tiles at each vertex,
the only solution is when all tiles are triangles, $(3^6)$;
but this is the regular tiling by triangles.
So we consider 4 tiles at each vertex.
If none of the tiles are triangles,
the only solution is four squares, $(4^4)$;
but this is the regular tiling by squares.
So we must have $(3.k.3.k)$.
The only solution is $k = 6$,
which is the trihexagonal tiling.
This tiling has two triangles and two hexagons alternating at each vertex.

On the other hand, if $\Gamma(\sT)$ acts transitively on facets,
we apply
\begin{lmm}[{\cite[\nopp 3.5.1]{grunbaum1986tilings}}]
If every tile of a normal tiling $\sT$ has $k$ vertices,
and these vertices have valences $j_1, \dotsc, j_k$,
then
\[
\sum_{i=1}^k \frac{j_i - 2}{j_i} = 2.
\]
\end{lmm}
Every facet has evenly many sides
(by Lemma~\ref{evenval} with $I=\{1,2\}$),
and the valence of each vertex alternates.
Clearly, this has the same solutions as before.
In a notation $[j_1\ldots j_k]$ giving, in cyclic order,
the valence of each vertex adjacent to a tile,
we have $[3^6]$, the regular tiling by hexagons;
$[4^4]$, the regular tiling by squares;
and $[3.6.3.6]$, the rhombille tiling.
The latter has rhombus tiles, with three tiles meeting at the obtuse corners,
and 6 tiles meeting at the acute corners.

\subsection{Tilings of \texorpdfstring{$\E^3$}{E\textasciicircum 3}}\label{space}
A tiling of $\E^3$ has rank 4. Hence, 
by Lemma~\ref{CombRegVerfsFacets},
if it has combinatorially regular facets and vertex figures,
it must be combinatorially regular.
So a combinatorially two-orbit tiling $\sT$ of $\E^3$ must have some tiles or some 
vertex figures from the list of two-orbit polyhedra.
By Corollary~$2'$,
the class of $\sT$ must be $2_{\{0,1,2\}}$ (tile-intransitive),
$2_{\{1,2,3\}}$ (vertex-intransitive),
or $2_{\{1,2\}}$ (fully transitive).
We consider these cases.

\subsubsection{Tile-intransitive}
In this case, $\sT$ is in class $2_{\{0,1,2\}}$
and the automorphism group is transitive on the vertices,
edges, and 2-faces of the tiling.
There are two different tile orbits,
and each tile must be combinatorially regular (since the orbit of a flag
is determined entirely by which type of tile it includes).
Thus, all the vertices have isomorphic vertex figures,
which must be a two-orbit polyhedron;
since there are two types of tile,
the vertex figure must be facet-intransitive,
i.e.\@ the cuboctahedron or the icosidodecahedron.

With the cuboctahedron as vertex figure, each vertex is 3-valent in some tiles,
and 4-valent in others.
The only regular polyhedron with 4-valent vertices is the octahedron;
the only regular polyhedron with 3-valent vertices and triangular faces (to match the octahedron)
is the tetrahedron. But the tiling built from tetrahedra and octahedra in this manner is the
tetrahedral-octahedral honeycomb, $\{3,\stak 3 4, 4\}$,
one of the (geometrically) two-orbit tilings.

With the icosidodecahedron as vertex figure,
each vertex is 3-valent in some tiles, and 5-valent in others.
The only regular polyhedron with 5-valent vertices is the icosahedron,
and the other tiles must be tetrahedra.
Such a tetrahedral-icosahedral tiling has type $\{3, \stak 3 5, 4\}$.
Indeed, a tiling can be built up in such a way, in hyperbolic space;
it is known as the alternated order-5 cubic honeycomb.
It can be carved out of a tiling by cubes, with 5 around each edges, $\{4,3,5\}$,
which is a regular tiling of hyperbolic space.
Inscribe a tetrahedron in each cube, so that tetrahedra in adjacent cubes alternate direction.
The shape left around a vertex which is not part of the tetrahedron is an icosahedron
(there are 20 cubes around each vertex in $\{4,3,5\}$.)

We show by contradiction that there is no normal tiling of $\E^3$ of this type.
Suppose $\sT$ is a normal tetrahedral-icosahedral tiling.
Divide each icosahedron of $\sT$ into twenty pyramids (over each of its 2-faces), and
adjoin each of these pyramids to the tetrahedron with which it shares the
2-face.
Thus we have partitioned $\E^3$ into tiles,
one for each tetrahedron in $\sT$,
consisting of the tetrahedron
with a pyramidal cap added to each of its 2-faces.
These tiles need not be convex,
but are isomorphic to cubes:
Each tile has six neighboring tiles,
with each of which it shares two triangular faces
of its pyramidal caps; we treat each such pair of triangular
faces as a single ``skew'' 4-gonal face.
Thus we get a tiling $\sC$ topologically isomorphic to the order-5
cubic tiling $\{4,3,5\}$. If we start with
a normal tiling, this one will be, also.
Say each tile contains a ball of radius $u$ and is contained in a ball
of radius $U$.
Then the number of tiles in a ball of radius $r$ is at most $r^3/u^3$.

Next consider a growing sequence of patches of the tiling $\sC$.
(For our purposes, a \emph{patch} can be defined
as any finite set of tiles of $\sC$ whose union is homeomorphic to a ball.)
Begin with a vertex of the tiling,
designated $A_0$.
Let $A_1$ consist of all
the tiles containing $A_0$, $A_2$ consist of all the tiles with nonempty intersection
with the union of $A_1$, and so on, so $A_{n+1}$ consists of all the tiles
with nonempty intersection with the union of $A_n$.

Let us call a tile of $A_n$ which has any 2-face on the boundary of $A_n$
a $k$-tile if it has $k$ 2-faces on the boundary of $A_n$.
By induction, we see that all tiles on the boundary are 1-tiles,
2-tiles, or 3-tiles; that every edge on the boundary of $A_n$
is contained in either 1 or 2 tiles of $A_n$; and that every vertex on the boundary of $A_n$
is contained in either 1, 2, or 5 tiles of $A_n$.
\begin{itemize}
\item
A tile not in $A_n$ which contains a 2-face in $A_n$
becomes a 1-tile of $A_{n+1}$.
Each of its four exposed edges is in two tiles of $A_{n+1}$.
See the leftmost example in Figure~\ref{newtiles}.
\item
A tile not in $A_n$ which contains only an edge in $A_n$
becomes a 2-tile of $A_{n+1}$.
Six of its exposed edges are in two tiles of $A_{n+1}$,
while one exposed edge is only in this tile.
See the middle example in Figure~\ref{newtiles}.
\item
A tile not in $A_n$ which contains only a vertex in $A_n$
becomes a 3-tile of $A_{n+1}$.
Six of its exposed edges are in two tiles of $A_{n+1}$,
while three exposed edges are only in this tile.
See the rightmost example in Figure~\ref{newtiles}.
\end{itemize}

\begin{figure}[h]
\begin{tikzpicture}[y={(3.85mm,3.85mm)},z={(0cm,1cm)},scale=2,every node/.style={color=black,fill=white,circle,solid,inner sep=.5pt},baseline=(lbl)]
\fill[gray!50] (0,0) -- (1,0) -- (1,1) -- (0,1) -- cycle;
\draw (0,0,0) -- (1,0,0) -- (1,1,0);
\draw[thin,gray] (1,1,0) -- (0,1,0) -- (0,0,0);
\draw (0,0,0) -- (.1,.1,1) node[draw] (a) {$\scriptstyle{5}$}
      (1,0,0) -- (.9,.1,1) node[draw] (b) {$\scriptstyle{5}$}
      (1,1,0) -- (.9,.9,1) node[draw] (c) {$\scriptstyle{5}$};
\draw[thin,gray] (0,1,0) -- (.1,.9,1) node[draw] (d) {$\scriptstyle{5}$};
\draw (a) -- node {$\scriptstyle{2}$} (b)
       -- node {$\scriptstyle{2}$} (c)
       -- node {$\scriptstyle{2}$} (d)
       -- node {$\scriptstyle{2}$} (a);
\node[rectangle] (lbl) at (.5,.5,-.5) {Over a 2-face in $A_n$};
\end{tikzpicture}
\hfil
\begin{tikzpicture}[y={(9mm,7mm)},z={(0cm,1.5cm)},x={(.9cm,0cm)},every node/.style={color=black,fill=white,circle,solid,inner sep=.5pt},baseline=(lbl)]
\draw[ultra thick] (1,0,0) -- (1,1,0);
\draw (1,0,0) -- (0,0,1) node[draw] (a) {$\scriptstyle{5}$}
      -- node {$\scriptstyle{2}$} (1,0,2) node[draw] (b) {$\scriptstyle{2}$}
      -- node {$\scriptstyle{2}$} (2,0,1) node[draw] (c) {$\scriptstyle{5}$} -- cycle;
\draw[thin,gray] (1,1,0) -- (0,1,1) node[draw] (d) {$\scriptstyle{5}$}
      -- node {$\scriptstyle{2}$} (1,1,2) node[draw] (e) {$\scriptstyle{2}$};
\draw (e) -- node {$\scriptstyle{2}$} (2,1,1) node[draw] (f) {$\scriptstyle{5}$} -- (1,1,0);
\draw (b) -- node {$\scriptstyle{1}$} (e)
      (c) -- node {$\scriptstyle{2}$} (f);
\draw[thin,gray] (a) -- node {$\scriptstyle{2}$} (d);
\node[rectangle] (lbl) at (1,.5,-.5) {Over an edge in $A_n$};
\end{tikzpicture}
\hfil
\begin{tikzpicture}[y={(5mm,3.5mm)},z={(0cm,-1.2cm)},every node/.style={color=black,fill=white,circle,solid,inner sep=.5pt},baseline=(lbl)]
\node[draw] (v) at (0,0,0) {$\scriptstyle{1}$};
\draw (v) -- node {$\scriptstyle{1}$} (1,0,1) node[draw] (a) {$\scriptstyle{2}$}
      (v) -- node {$\scriptstyle{1}$} (-.5,-.866,1) node[draw] (c) {$\scriptstyle{2}$};
\draw[thin,gray] (v) -- node {$\scriptstyle{1}$} (-.5,.866,1) node[draw] (b) {$\scriptstyle{2}$}      
      (.5,.866,2) node[draw] (d) {$\scriptstyle{5}$} -- node {$\scriptstyle{2}$} 
      (b) -- node {$\scriptstyle{2}$} (-1,0,2) node[draw] (e) {$\scriptstyle{5}$};
\draw (e) -- node {$\scriptstyle{2}$} 
      (c) -- node {$\scriptstyle{2}$} (.5,-.866,2) node[draw] (f) {$\scriptstyle{5}$} -- node {$\scriptstyle{2}$} 
      (a) -- node {$\scriptstyle{2}$}  (d) -- 
      (0,0,3) node[fill=black,inner sep=2pt] (w) {}
      (e) -- (w)
      (f) -- (w);
\node[rectangle] (lbl) at (0,0,3.4) {Over a vertex in $A_n$};
\end{tikzpicture}
\caption{Tiles of $A_{n+1} \enleve A_n$. Vertices and edges on the boundary of $A_{n+1}$ are labeled by the number of tiles of $A_{n+1}$ containing them. Elements belonging to $A_n$ are darkened.}\label{newtiles}
\end{figure}

Let $a_n$ be the number of 1-tiles in $A_n$,
$b_n$ be the number of 2-tiles, and $c_n$ be the number of 3-tiles.
The 1-tiles in $A_{n+1}$ are those tiles which share a 2-face with
some $k$-tile of $A_n$,
so we have
\[
a_{n+1} = a_n + 2 b_n + 3 c_n.
\]
The 2-tiles in $A_{n+1}$ are added above edges on the boundary
of $A_n$. One such 2-tile is added above each edge contained
in two tiles of $A_n$, and two such 2-tiles are added above
each edge contained in a unique tile of $A_n$.
Counting the number of edges of each type in the boundary of $A_n$ yields
\begin{align*}
b_{n+1} &= \frac{4 a_n + 6 b_n + 6 c_n}{2} + 2 (b_n + 3 c_n) \\
&= 2 a_n + 5 b_n + 9 c_n.
\end{align*}
The 3-tiles in $A_{n+1}$ are added above vertices on the boundary of $A_n$.
A vertex contained in five tiles of $A_n$
is also contained in five 1-tiles of $A_{n+1}$, added above
the five incident 2-faces in the boundary of $A_n$,
and in five 2-tiles of $A_{n+1}$ added above the five incident edges
in the boundary of $A_n$,
so we need to add five more 3-tiles to include all 20 of the incident tiles.
A vertex contained in two tiles of $A_n$
is incident to eight 3-tiles of $A_{n+1}$,
and a vertex in a unique tile of $A_n$
is incident to ten 3-tiles of $A_{n+1}$.
Hence
\begin{align*}
c_{n+1} &= 5 \frac{4 a_n + 4 b_n + 3 c_n}{5} + 8 \frac{2 b_n + 3 c_n}{2} + 10 c_n \\
&= 4 a_n + 12 b_n + 25 c_n.
\end{align*}

Thus we have an equation for the number of $k$-tiles on the boundary
of each patch $A_n$,
beginning with the patch $A_1$ consisting of the 20 tiles incident to $A_0$:
\[
\begin{bmatrix}
a_n \\ b_n \\ c_n
\end{bmatrix}
=
\begin{bmatrix}
1 & 2 & 3 \\
2 & 5 & 9 \\
4 & 12 & 25
\end{bmatrix}^{n-1}
\begin{bmatrix}
0 \\ 0 \\ 20
\end{bmatrix}.
\]
This matrix is diagonalizable, making it straightforward to solve
for the total number of tiles in the patch $A_n$:
\[
\abs{A_n} = 
\frac 5 7 \left(\frac{9}{2\sqrt{14}}\bigl((15 + 4\sqrt{14})^n - (15 - 4 \sqrt{14})^n\bigr) - 8n\right).
\]
This is exponential in $n$;
the number of tiles increases by a factor of roughly 30 in each
successive patch.
Now consider a ball centered at $A_0$ with radius $2nU$.
This ball contains the patch $A_n$,
but the number of tiles in the ball is at most
$(2nU)^3/u^3$.
An exponential function of $n$ cannot remain bounded by a cubic
function of $n$,
so there must be some $n$ such that $\abs{A_n} > (2nU)^3/u^3$,
a contradiction.

Therefore, no normal tiling of $\E^3$ has type $\{4,3,5\}$, even 
allowing non-convex tiles.
So no tetrahedral-icosahedral tiling
of type $\{3, \stak 3 5, 4\}$ can be normal either.
On the other hand, there seems to be no obstruction to constructing
(non-normal) LFC tilings of these types.

\subsubsection{Vertex-intransitive}
In this case, the orbit of a flag is determined by the vertex it contains.
So the vertex figures are combinatorially regular. The tiles are two-orbit vertex-intransitive polyhedra, i.e.\@ the rhombic dodecahedron or rhombic triacontahedron.

With the rhombic dodecahedron, a vertex which is incident
to 4 edges in a given tile
has a vertex figure with a square face; hence the vertex figure is a cube.
Hence each edge incident to the vertex is in 3 tiles.
A vertex which is incident to 3 edges in a given tile
has a vertex figure with triangular faces.
Every edge of the tiling is incident to one vertex of each type,
hence is in 3 tiles,
so the second type of vertex figure must be a tetrahedron.
Rhombic dodecahedra put together in this way form the rhombic dodecahedral honeycomb $\{4,\stak 3 4, 3\}$, one of the (geometrically) two-orbit tilings.

With the rhombic triacontahedron as tile, any vertex
which is incident to five edges in a given tile
has a pentagon in its vertex figure; hence its vertex figure
is a combinatorially regular dodecahedron.
Every edge is incident to one vertex of this type,
so every edge is in three tiles.
Thus the other vertices, which are incident to three edges in each tile,
have tetrahedra for vertex figures.
This potential tiling  has type $\{4, \stak 3 5, 3\}$ and is dual to the tetrahedral-icosahedral tiling discussed above. Like that one, this tiling can be realized in hyperbolic space, with a two-orbit symmetry group.
For any normal tiling there is a dual tiling which is also normal
(but the tiles of which are not necessarily convex).
Hence, if a normal tiling of type $\{4, \stak 3 5, 3\}$ existed,
we could find a normal tiling of type $\{3, \stak 3 5, 4\}$, which we know to be impossible.

\subsubsection{Class $2_{\{1,2\}}$}
These is the same class discussed in Section~\ref{4polys},
and the same considerations establish that
we have either a cuboctahedron vertex figure with rhombic dodecahedra as tiles, type $\{4,\stak 3 4, 4\}$,
or an icosidodecahedron vertex figure with rhombic triacontahedra as tiles,
type $\{4, \stak 3 5, 4\}$.
(We note that the cuboctahedron and icosidodecahedron
are non-tiles, meaning there cannot be any tiling of $\E^3$ using only tiles
isomorphic to these; see \cite{schulte1985existence}.)

Perhaps these types can be realized as LFC tilings.
However, there is no such normal tiling.
Essentially the same proof as in Section~\ref{4polys} applies,
along with the Normality Lemma \cite[45]{schattschneider1997tilings},
which says that in a normal tiling, the ratio of
(the number of tiles that meet the boundary of a spherical patch of the tiling)
to (the number of tiles in the patch) goes to zero as the radius of the patch grows.
The two methods of counting internal angles of 2-faces in Section~\ref{4polys} hold for all the faces in the interior of a given patch. Discrepancies occur only at tiles on the boundary, where a vertex is not surrounded by all the 2-faces incident to it in the tiling. Taking the limit as the patch grows, the discrepancies go to zero and the inequality remains.
The details are too tedious to include here.

With these ruled out, we have established

\begin{thm}
Every normal combinatorially two-orbit tiling of $\E^3$
is isomorphic to one of the (geometrically) two-orbit tilings:
either the tetrahedral-octahedral honeycomb
or its dual, the rhombic dodecahedral honeycomb.
\end{thm}

\section{Open Questions}
\begin{qst}
Is a combinatorially regular LFC tiling of $\E^d$, $d \geq 3$,
necessarily isomorphic to a regular tiling of $\E^d$?
(Except for $d = 4$, this says that any combinatorially regular tiling
is isomorphic to the tiling by $d$-cubes.)
\end{qst}
The answer is probably \emph{no}, but the author does not know a counterexample.

\begin{qst}
Is a combinatorially regular normal LFC tiling of $\E^d$ necessarily isomorphic to a regular tiling of $\E^d$?
\end{qst}
The answer is probably \emph{yes}, but the author knows a proof only for the cases $d \leq 2$.

\begin{qst}
Are there combinatorially two-orbit LFC tilings of $\E^3$ not isomorphic to any two-orbit tiling?
\end{qst}
Any such tiling would have one of the previously discussed types 
$\{3,\stak{3}{5},4\}$, $\{4,\stak{3}{5},3\}$, $\{4,\stak{3}{4},4\}$, or $\{4,\stak{3}{5},4\}$.
The author believes that non-normal tilings of these types probably do exist.

For results in these directions, as well as other open questions of this type,
see \cite{schulte2010combinatorial}.

\printbibliography

\end{document}